\documentclass[reqno]{amsart}

\usepackage{amsmath,amssymb,amsthm}

\usepackage{tikz}

\usepackage{caption}
\usepackage{graphicx}
\usepackage{graphics}

\newtheorem{theorem}{Theorem}

\newtheorem{lemma}{Lemma}
\newtheorem{proposition}{Proposition}

\begin{document}

\title[]{Small Gaps in the Ulam sequence}

\author[]{Fran\c{c}ois Cl\'ement}
\author[]{Stefan Steinerberger}
\address{Department of Mathematics, University of Washington, Seattle}
\email{fclement@uw.edu} 
\email{steinerb@uw.edu}

\begin{abstract}
The Ulam sequence, described by Stanislaw Ulam in the 1960s, starts $1,2$ and then iteratively adds the smallest integer that can be uniquely written as the sum of two distinct earlier terms: this gives $1,2,3,4,6,8,11,\dots$. Already in 1972 the great French poet Raymond Queneau wrote that it `gives an impression of great irregularity'. This irregularity appears to have a lot of structure which has inspired a great deal of work; nonetheless, very little is rigorously proven. We improve the best upper bound on its growth and show that at least \textit{some} small gaps have to exist: for some $c>0$ and all $n \in \mathbb{N}$
$$ \min_{1 \leq k \leq n} \frac{a_{k+1}}{a_k} \leq 1 + c\frac{\log{n}}{n}.$$
\end{abstract}

 \maketitle

\section{Introduction and Results}
\subsection{Introduction.}
The origins of the Ulam sequence are mysterious. It is mentioned in a 1962 paper of Ulam \cite{ulam0}. Two years later, Ulam \cite{ulam} writes
\begin{quote}
 One can consider a rule for growth of patterns  -- in one dimension it would be
merely a rule for obtaining successive integers. [...] Another one-dimensional sequence was studied recently. This
one was defined purely additively: one starts, say, with integers 1, 2 and considers in turn all integers which are obtainable as the sum of two different ones
previously defined but only if they are so expressible in a unique way. The
sequence would start as follows:
$$1,2,3,4,6,8,11,13,\dots$$
\end{quote}
No additional motivation is given and it is not at all clear why the arising sequence
should be of interest. A few years later, Raymond Queneau, one of the most famous French poets of the twentieth century (and a co-founder of Oulipo), published the article \textit{Sur les suites s-additives} \cite{q0} (CRAS, 1968) followed by a more extensive article with the same name \cite{q} (J. Comb. Theory, 1972). Queneau considers $s-$additive sequences, sequences of integers where each new term is the sum of exactly $s$ previous terms. He shows that these can exhibit very interesting behavior. For $s=1$, he remarks that we get the sequence `studied by S. M. Ulam. It gives the impression of great irregularity' \cite{q}. The Ulam sequence and, more generally, $s-$additive sequences were then studied by  Cassaigne-Finch \cite{cas}, Finch \cite{finch1, finch3, finch2, finch0}, Mauldin-Ulam \cite{mauldin}, Recaman \cite{recaman}, Schmerl-Spiegel \cite{schm} and others.
Around ten years ago, one of the authors discovered (numerically) that the `great irregularity' already observed by Queneau seems to exhibit a great degree of regularity: there seems to exist a real number $\alpha \sim 2.57145\dots$ so that the sequence $\alpha a_n \mod 2\pi$ exhibits a most curious distribution \cite{stein}. This inspired subsequent work \cite{adut,angelo, bade, hinman, hinman2, kravitz, kuca, mandel, sen} but remains unexplained to this day.

\subsection{Results} Despite all of this, \textit{very} little is actually rigorously known about the Ulam sequence. It is known that the sequence is infinite: if it were finite, $a_1, \dots, a_n$, then $a_{n-1} + a_n$ has a unique representation as the sum of two distinct earlier terms. This proves $a_{n+1} \leq a_n + a_{n-1}$ which can be used to show that $a_n \leq F_n$ where $F_n$ is the $n-$th Fibonacci number. A further refinement is possible.

\begin{proposition}[R. B. Eggleton, \cite{recaman}] We have
$$ a_{n+1} \leq a_n + a_{n-2}.$$
This implies $a_n \leq 1.466^n$ where $1.46\dots$ is the root of $x^3 - x^2 - 1=0$.
\end{proposition}
\begin{proof} The number $a_n + a_{n-2}$ can be written as the sum of two distinct earlier terms. It remains to show that there is no other representation. If $a_{n-2} + a_n = a_i + a_j$ with $1 \leq i < j \leq n$, then $j$ cannot be $n$ and $j$ cannot be smaller than $n-2$. If $j=n-1$, then $i \leq n-2$ which is a contradiction.
\end{proof}
To the best of our knowledge, this is the only rigorously proven result about the growth of the Ulam sequence. Our first contribution is an improvement.
\begin{theorem}
    We have, for all $n$ sufficiently large,
    $$ a_n \leq 1.454^n.$$
\end{theorem}
This small improvement requires a surprising amount of work; the main idea is that Eggleton's bound cannot be sharp twice in a row. This leads us to consider two near-Eggleton recursions and the potential order in which they could be applied; after some reductions, we are faced with a variation of the joint spectral radius problem with an additional restriction on allowed words; we circumvent the problem by using submultiplicativity of the operator norm to reduce it to a finite case distinction. Studying more cases would lead to a (slightly) improved constant, however, we also prove that our argument in its current form cannot prove $a_n \leq \rho^n$ for $\rho < 1.4146$ (which might in principle be achievable with our approach and more computation). We believe that there might be a limit to how much one can obtain via these types of arguments. It feels unlikely that one could prove subexponential growth this way. 
Numerical experiments suggest that the sequence grows \textit{linearly}. The massive gulf between what is known and what seems to be true is a \textit{very} unsatisfying state of affairs. We will prove that at least some small gaps exist.  

\begin{theorem} There exists $c>0$ such that for all $n \in \mathbb{N}_{\geq 2}$
$$ \min_{1 \leq k \leq n} \frac{a_{k+1}}{a_k} \leq 1 +  c \frac{\log{n}}{n}.$$
\end{theorem}
The proof shows that $c=7$ would work for $n$ sufficiently large. The inequality appears to be close to sharp: if $a_k$ grows linearly (as is expected), then the result would be sharp up to the logarithmic factor. Needless to say, the result does not exclude exponential growth. It feels appropriate, this being the subject of his mathematical publications, to conclude with Queneau (from his \textit{Morale \'el\'ementaire})
\begin{quote}\textit{S'il levait le nez, il verrait au-del\`a des nuages les v\'erit\'es arithm\'etiques, mais ce n'est que la nuit qu'il sombre dans le vertige des \'etoiles.}
\end{quote}

\section{Proof of Theorem 1}
We start with a bit of terminology. We always have Eggleton's bound 
$$ a_{n+1} \leq a_n + a_{n-2} \qquad \qquad \qquad (E)$$
at our disposal. The entire idea behind the argument is that Eggleton's bound is a worst case scenario and we will try to avoid using it as much as possible.
The next two lexicographic terms  are $a_n + a_{n-3}$ (which we refer to as Type I) and $a_{n-1} + a_{n-2}$ (which we refer to as Type II). These are some of the largest sums that can be attained in the next step: Eggleton's bound being the largest and Type I and Type II being incomparable between each other. 

\begin{lemma}
    Let $1 \leq i < j \leq n$. If $a_i + a_j < a_n + a_{n-2}$, then 
    $$ a_{i} + a_j \leq \max\left\{ a_n + a_{n-3}, a_{n-1} + a_{n-2}\right\}$$
    with equality if and only if $(i,j) = (n-3, n)$ or $(i,j) = (n-2, n-1)$. If $a_{n+1} = a_n + a_{n-2}$, then Type I and Type II coincide, i.e. $ a_n + a_{n-3} = a_{n-1} + a_{n-2}$.  
\end{lemma}
\begin{proof}
    If $j=n$, then $a_i + a_j < a_n + a_{n-2}$ forces $i \leq n-3$ and the result follows. If $j = n-1$, then $i \leq n-2$ and the result follows. If $j\leq n-2$, then $i \leq n-3$ and the result follows.  If Eggleton's bound is attained, then Type I and Type II have to coincide: if they were different numbers, then Type I, $a_n + a_{n-3}$, would have a unique representation which can be seen as follows: suppose
    $$ a_n + a_{n-3} = a_i + a_j.$$
    If $j = n$, then this forces $i = n-3$ which is the same representation. If $j = n-1$, then $i = n-2$, which is Type II which is different by assumption. But if $j =n-2$, we are only left with $i=n-3$ which is too small. Therefore $a_n + a_{n-3}$ has a unique representation and therefore $a_{n+1} \leq a_n + a_{n-3}$. Then $a_{n+1} \leq a_n + a_{n-3} < a_n + a_{n-2}$ which is a contradiction.
\end{proof}

\begin{lemma}
    The Eggleton bound cannot be attained twice in a row.
\end{lemma}
\begin{proof}
    
Suppose this is false: after two applications of Eggleton's bound, we have
$$ a_{n-2}, a_{n-1}, a_n, \quad a_n + a_{n-2}, \quad a_n + a_{n-1} + a_{n-2}.$$
Since Eggleton's bound is exact both times, we have from Lemma 1 that the Type I and Type II bound coincide each time. This means that, for the second application,
$$ a_{n-2} + (a_n + a_{n-2}) = a_n + a_{n-1}.$$
Then $2 a_{n-2} = a_{n-1}$ contradicting Eggleton bound's $a_{n-1} \leq a_{n-2} + a_{n-4}$.
\end{proof}

\begin{proof}[Proof of Theorem 1.]
We know that the Ulam sequence $a_1, a_2, \dots$ exists and construct a new sequence $b_1, b_2, \dots$ such that $a_n \leq b_n$. We set $b_n = a_n$ for $1 \leq n \leq 5$. After that, we construct $b_{n+1}$ as follows.  The Ulam element $a_{n+1}$ is either obtained via Eggleton ($a_{n+1} = a_{n} + a_{n-2}$), via Type I ($a_{n+1} = a_n + a_{n-3}$), via Type II ($a_{n+1} = a_{n-1} + a_{n-2}$) or via some other method $a_{n+1} = a_i + a_j$. If it is any of the first three, then we use the exact same recursion for $b_{n+1}$ which then guarantees $a_{n+1} \leq b_{n+1}$. If it is none of these three recursions, then we know (Lemma 1) that $a_{n+1} = a_i + a_j$ is either bounded by Type 1 or Type 2 and we use that corresponding recursion (if both Type I and Type II are upper bounds, then either choice is admissible) to define $b_{n+1}$. We also note that $b_n$ only uses Eggleton when $a_n$ uses Eggleton and this cannot happen twice in a row. The remaining question is: how quickly can such a sequence $(b_n)_{n=1}^{\infty}$ grow?
At this point, we start using the language of matrices. The language of matrices allows us to write an application of these bounds as 
$$\begin{pmatrix} b_{n+1} \\ b_n \\ b_{n-1} \\ b_{n-2}   \end{pmatrix} = \begin{pmatrix}
    1 & 0 & 0 & 1  \\
    1 & 0 & 0 & 0 \\
  0&  1 & 0 & 0  \\
   0&  0 & 1 & 0 
\end{pmatrix}  \begin{pmatrix} b_{n} \\ b_{n-1} \\ b_{n-2} \\ b_{n-3}  \end{pmatrix} \qquad (T_1 = \mbox{Type I)} $$
and
$$\begin{pmatrix} b_{n+1} \\ b_n \\ b_{n-1} \\ b_{n-2}  \end{pmatrix} = \begin{pmatrix}
    0 & 1 & 1 & 0  \\
    1 & 0 & 0 & 0 \\
  0&  1 & 0 & 0  \\
   0&  0 & 1 & 0  \\
\end{pmatrix}  \begin{pmatrix} b_{n} \\ b_{n-1} \\ b_{n-2} \\ b_{n-3} \end{pmatrix} \qquad (T_2 =\mbox{Type II)}$$
and
$$ \begin{pmatrix} b_{n+1} \\ b_n \\ b_{n-1} \\ b_{n-2}  \end{pmatrix} = \begin{pmatrix}
    1 & 0 & 1 & 0  \\
    1 & 0 & 0 & 0 \\
  0&  1 & 0 & 0  \\
   0&  0 & 1 & 0  
\end{pmatrix}  \begin{pmatrix} b_{n} \\ b_{n-1} \\ b_{n-2} \\ b_{n-3} \end{pmatrix} \qquad (T_3 =\mbox{Eggleton}).$$

We refer to these matrices as $T_1, T_2$ and $T_3$. The remaining problem is to understand the possible behavior of the growth of the operator norm of products of these forms, more precisely, we are interested in bounds along the lines of
$$ \left\|   \prod_{i=1}^{n} T_{k_i} \right\| \leq C^n \qquad \mbox{where}~k_i \in \left\{1,2,3\right\}.$$
This leads us naturally to the \textit{joint spectral problem} \cite{rota} with the notable difference that not all sequences $(k_i)$ are admissible (because $k_i = 3 = k_{i+1}$ is not admissible, Eggleton cannot arise twice in a row). The joint spectral problem is known to be difficult even without such constraints \cite{bl} and we do not aim to resolve it exactly and will instead focus on approximate solutions. Fixing the segment's length $L$, we can always use the submultiplicativity of the operator norm to bound, for $n \geq L$,
$$  \left\|   \prod_{i=1}^{n} T_{k_i} \right\|  \leq C_L \left(  \max_{\mbox{\tiny admissible word} \atop \mbox{\tiny of length L}} \quad \left\| \prod_{i=1}^{L} T_{k_i} \right\| \right)^{n/L},$$
    where the maximum ranges over all admissible words on $T_1, T_2, T_3$ of length $L$.
    Therefore, for all $L \in \mathbb{N}$,
    $$ C \leq \max_{\mbox{\tiny admissible word} \atop \mbox{\tiny of length L}} \quad \left\| \prod_{i=1}^{n} T_{k_i} \right\|^{1/L}.$$
Using this with $L=15$ and checking all admissible words gives 
$$ C \leq 1.4539\dots$$
with the extremal word being $ W =  (T_3 T_1)^3   (T_1 T_3)^3 T_2.$
One naturally wonders how far this is from the truth. Here we use the Gelfand formula \cite{gelfand} stating that the spectral radius is connected to the asymptotic behavior via
$$ \rho(A) = \lim_{n \rightarrow \infty} \|A^n\|^{1/n}.$$
For any admissible word $W = \prod_{i=1}^{L} T_{n_i}$ of length $L$ (whose concatenation with itself remains an admissible word)  we cannot hope to do better than $ \rho(W)^{1/L}$. The word
$W = T_3 T_1^{2}$ satisfies $\rho(W)^{1/3} \sim 1.4146\dots$ (the $1/3-$power of a root of a cubic polynomial).
\end{proof}
 
It is worth noting that $\rho(T_3 T_1)^{1/2} = \sqrt{2}$ is just a tiny bit smaller than $\rho(T_3 T_1^2)^{1/3} $. This leads to a number of near-minimizers when truncating after $L=14$ words, the extremal structure is difficult to discern at that stage. It is not inconceivable and maybe even plausible that the optimal solution is asymptotically given by powers of $T_3 T_1^2$ which would give the asymptotic rate $\rho(T_3 T_1^2)^{1/3} \sim 1.4146$.

\section{Proof of Theorem 2}

\begin{proof}
Let us fix $n$ and define $\delta$ by
$$  \min_{1 \leq k \leq n} \frac{a_{k+1}}{a_k} = 1 + \delta.$$
We know, since $a_2 = 2$ and $a_3 = 3$, that $\delta \leq 1/2$ once $n \geq 3$. As a first consequence, for all $1 \leq k \leq n$ we have that $a_{k+1} \geq (1+\delta) a_k$. Applying this iteratively, we see that, for $1 \leq \ell \leq n-1$,
$$ a_{n-\ell} \leq \frac{a_n}{(1+\delta)^{\ell}}.$$
This inequality implies 
$$ \# \left\{1 \leq k \leq n: a_k \geq \frac{a_n}{(1+\delta)^{\ell}} \right\} \leq \ell + 1$$
which, in turn, can be rewritten as, for all $1 \leq x \leq a_n$
 $$ \# \left\{1 \leq k \leq n: a_k \geq x \right\} \leq \frac{\log{(a_n/x)}}{\log{(1+\delta)}} + 1 \leq \frac{2}{\delta} \log{\left( \frac{a_n}{x} \right)} + 1,$$
 where the second inequality used $0 < \delta < 1$. There are many Ulam numbers smaller than $\delta a_n/2$ since
$$  \# \left\{1 \leq k \leq n: a_k \geq \frac{\delta a_n}{2} \right\} \leq \frac{2}{\delta} \log\left(\frac{2}{\delta}\right) + 1$$
and therefore
$$  \# \left\{1 \leq k \leq n: a_k \leq \frac{\delta a_n}{2} \right\} \geq n - \frac{2}{\delta} \log\left(\frac{2}{\delta}\right) - 1.$$
Thus, there are many small Ulam numbers that we could add to $a_n$ that yield a sum that is not much larger than $a_n$. Rewriting the previous bound
$$ \# \left\{ 1 \leq k < n:  a_n + a_k \in \left[ a_n, a_n + \frac{\delta}{2} a_n\right] \right\}  \geq n - \frac{2}{\delta} \log\left(\frac{2}{\delta}\right) - 1.$$
However, none of these numbers are in the sequence since $a_{n+1} \geq (1+\delta) a_n$ is much larger. The only way this is possible is if for each of these sums there exists an alternative representation $a_n + a_i = a_j + a_k$ for some $j < k \leq n=1$. We will argue that for $\delta$ large, the number of sums of the form $a_n + a_i$ with $a_i \leq a_n \delta/2$ (already bounded above) is larger than the number of (relevant) sums $a_j + a_k$: then, by pigeonholing, there must exist one `unblocked' sum $a_n + a_i$ which arises only once; this sum would then be an upper bound on $a_{n+1}$ which is a contradiction. This requires us to bound 
$$X = \# \left\{ (j,k) \in \mathbb{N}^2: j <k ~\mbox{and}~ a_n \leq  a_j + a_k \leq a_n + \frac{\delta}{2} a_n \right\}. $$
We may think of the sum $a_j + a_k$ as the sum of two terms where $a_k$ is `large' and $a_j$ is `small'. The `large' one has to be at least $\geq a_n/2$ and a `smaller' one that has to be in a precise regime. Using the estimate above
$$ \frac{a_n}{(1+\delta)^{\ell}} \geq a_{n-\ell} \geq \frac{a_n}{2} \qquad \mbox{one deduces} \qquad \ell \leq \frac{2 \log{2}}{\delta}.$$
Therefore
\begin{align*}
X &\leq \sum_{\ell=1}^{2 \log{(2)}/\delta} \# \left\{ 1 \leq k \leq n:  a_n - a_{n - \ell} \leq a_k \leq a_n - a_{n- \ell} + \frac{\delta}{2} a_n \right\}.
\end{align*}
By the $\delta-$gap assumption, we have for any interval $1 \leq a < b \leq a_n$ that
$$ \# \left\{ 1 \leq k \leq n: a \leq a_k \leq b\right\} \leq \frac{\log{(b/a)}}{\log(1+\delta)} + 1 \leq \frac{2}{\delta}\log{\left(\frac{b}{a}\right)} + 1.$$
Before combining these inequalities, we simplify the algebra. First,
\begin{align*}
   \frac{a_n - a_{n- \ell} + \frac{\delta}{2} a_n}{a_n - a_{n - \ell} } &= 1 + \frac{\delta}{2} \frac{a_n}{a_n - a_{n-\ell}}.
\end{align*}
Using the inequality from above,
\begin{align*}
\frac{a_n}{a_n - a_{n-\ell}} &\leq \frac{a_n}{a_n - \frac{a_n}{(1+\delta)^{\ell}}} =  \frac{1}{1 - \frac{1}{(1+\delta)^{\ell}}} = \frac{(1+\delta)^{\ell}}{(1+\delta)^{\ell} - 1}.
\end{align*}
Since $\ell \leq 2 \log{(2)} /\delta$, we have
$$ (1 + \delta)^{\frac{2 \log{2}}{\delta}} \leq \left[ (1+\delta)^{\frac{1}{\delta}} \right]^{2 \log{2}} \leq e^{2 \log{(2)}} =4.$$
We also have Bernoulli's inequality $(1+\delta)^{\ell} \geq 1 + \ell \delta$ which implies
$$ \frac{a_n}{a_n - a_{n-\ell}} \leq \frac{(1+\delta)^{\ell}}{(1+\delta)^{\ell} - 1} \leq \frac{4}{\delta \ell}$$
and thus
$$ \frac{2}{\delta} \log\left( 1 + \frac{\delta}{2} \frac{4}{\delta \ell}  \right) \leq \frac{2}{\delta} \log\left( 1 + \frac{2}{\ell} \right) \leq \frac{4}{\delta \ell}.$$
Collecting all these bounds, we get
\begin{align*}
X &\leq \sum_{\ell=1}^{2 \log{(2)}/\delta} \# \left\{ 1 \leq k \leq n:  a_n - a_{n - \ell} \leq a_k \leq a_n - a_{n- \ell} + \frac{\delta}{2} a_n \right\} \\
&\leq \sum_{\ell=1}^{2 \log{(2)}/\delta} \left( 1 + \frac{4}{\delta \ell}\right) = \frac{2\log{(2)}}{\delta} + \frac{4}{\delta}  \sum_{\ell=1}^{2 \log{(2)}/\delta} \frac{1}{\ell} \leq \frac{2\log{(2)}}{\delta} + \frac{4}{\delta} \log\left( \frac{10}{\delta}\right).
\end{align*} 
We require
$$  \# \left\{ 1 \leq k < n:  a_n + a_k \in \left[ a_n, a_n + \frac{\delta}{2} a_n\right] \right\} \leq X$$
and with the lower bound on the first set derived above and the upper bound on $X$ just derived
$$   n - \frac{2}{\delta} \log\left(\frac{2}{\delta}\right) - 1 \leq X \leq \frac{2\log{(2)}}{\delta} + \frac{4}{\delta} \log\left( \frac{10}{\delta}\right).$$
As $n$ becomes large, this forces
$$ \delta \leq c\frac{\log{n}}{n}$$
and a short computation shows that $c = 7$ is admissible for $n$ sufficiently large.
\end{proof}

  The proof both suggests a certain type of strategy and also the difficulty faced by that strategy. If the sequence were to grow quickly, say exponentially, then $a_k \ll a_n$ is true for `almost all' $1 \leq k \leq n$. This means that there are many `candidate sums' $a_n + a_k$ that should potentially be in the sequence and that are not much larger than $a_n$ (leading to slow growth). Moreover, there should be few $a_i + a_j$ sums in the same range and they should not be able to `block' all the $a_n + a_k$. The main obstruction appears to be the one shown in Figure 2: the scenario where elements of the sequence arise in groups that are clumped together followed by a huge gap to the next element. It is not clear to us how to exclude this scenario.

\begin{center}
    \begin{figure}[h!]
        \centering
\begin{tikzpicture}
    \draw [ultra thick] (0,0) -- (8,0);
    \draw [ultra thick] (4, -0.1) -- (4,0.1);
        \draw [ultra thick] (4.2, -0.1) -- (4.2,0.1);
    \draw [ultra thick] (3.9, -0.1) -- (3.9,0.1);
    \draw [ultra thick] (4.15, -0.1) -- (4.15,0.1);
    \draw [ultra thick] (4.8, -0.1) -- (4.8,0.1);
    \draw [ultra thick] (4.82, -0.1) -- (4.82,0.1);
    \draw [ultra thick] (4.93, -0.1) -- (4.93,0.1);
    \draw [ultra thick] (5.04, -0.1) -- (5.04,0.1);
    \draw [ultra thick] (7.1, -0.1) -- (7.1,0.1);
    \draw [ultra thick] (7.23, -0.1) -- (7.23,0.1);
    \draw [ultra thick] (7.15, -0.1) -- (7.15,0.1);
    \draw [ultra thick] (1.4, -0.1) -- (1.4,0.1);
    \draw [ultra thick] (1.45, -0.1) -- (1.45,0.1);
    \draw [ultra thick] (1.6, -0.1) -- (1.6,0.1);
\end{tikzpicture}       
        \label{fig:}
        \caption{`Clumps' followed by a big jump.}
    \end{figure}
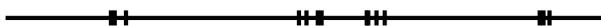
\end{center}

\end{document}